\documentclass[11pt,a4paper,reqno]{amsart}
\usepackage{amssymb,amsmath,amsthm,graphicx,color}
\theoremstyle{plain}
\newtheorem{theorem}{Theorem}[section]
\newtheorem{proposition}[theorem]{Proposition}

\newtheorem{lemma}[theorem]{Lemma}

\theoremstyle{remark}
\newtheorem{remark}[theorem]{Remark}

\usepackage
[bookmarks=true,
   bookmarksnumbered=false,
   bookmarkstype=toc]
   {hyperref}

\allowdisplaybreaks[4]
\numberwithin{equation}{section}

\newcommand{\C}{\mathbb{C}}
\newcommand{\R}{\mathbb{R}}

\def\({\left(}
\def\){\right)}
\def\<{\left\langle}
\def\>{\right\rangle}
\def\le{\leqslant}
\def\ge{\geqslant}

\newcommand{\rre}{\mathbb{R}}
\newcommand{\pt}{\partial}

\begin{document}
\title[On NLS with higher order anisotropic dispersion]
{Long range scattering for the nonlinear\\
Schr\"{o}dinger equation with higher order\\
anisotropic dispersion in two dimensions}

\author[J.-C. Saut]{Jean-Claude Saut}

\address{Laboratoire de Math\' ematiques, CNRS and Universit\' e Paris-Sud\\
91405 Orsay, France}
\email{jean-claude.saut@u-psud.fr}

\author[J.Segata]
{Jun-ichi Segata}

\address{Mathematical Institute, Tohoku University\\
6-3, Aoba, Aramaki, Aoba-ku, Sendai 980-8578, Japan}
\email{segata@m.tohoku.ac.jp}

\subjclass[2000]{Primary 35Q55
; Secondary 35B40}

\keywords{Schr\"{o}dinger equation with higher order dispersion,  
scattering problem}


\begin{abstract}
This paper is a continuation of our previous study \cite{SS} on 
the long time behavior of solution to the nonlinear Schr\"{o}dinger 
equation with higher order anisotropic dispersion (4NLS). We prove 
the long range scattering for (4NLS) with the quadratic nonlinearity 
in two dimensions. More precisely, for a given asymptotic profile 
$u_{+}$, we construct a solution to (4NLS) which converges to 
$u_{+}$ as $t\to\infty$, where $u_{+}$ is given 
by the leading term of the solution to the linearized equation of 
(4NLS) with a logarithmic phase correction.

\end{abstract}

\maketitle

\section{Introduction}

This paper is a continuation of our previous study \cite{SS} on 
the long time behavior of solution to the nonlinear Schr\"{o}dinger 
equation with higher order anisotropic dispersion:
\begin{eqnarray}
i\pt_tu+\frac12\Delta u-\frac{1}{4}\pt_{x_{1}}^4u=\lambda|u|^{p-1}u,
\qquad t>0,x\in\rre^{d},
\label{4NLS}
\end{eqnarray}
where $u:\rre\times\rre^{d}\to\C$ is an unknown function and $p>1$. 
Equation (\ref{4NLS}) arises in nonlinear optics to model the propagation
of ultrashort laser  pulses in a 
medium  with  anomalous  time-dispersion  in  the  presence  of
fourth-order time-dispersion   (see \cite {Ber,FIS,WF} 
and the references therein). It also arises in models of propagation 
in fiber arrays (see \cite{ADRT, FIS1}). 
The readers can consult \cite{B} for the well-posedness of (\ref{4NLS}), 
and existence/non-existence and qualitative properties results of solitary wave 
solutions for (\ref{4NLS}). 

In this paper, we consider the scattering problem for (\ref{4NLS}). 
Since the solution to the linearized equation of (\ref{4NLS}) decays like  
$O(t^{-d/2})$ in $L^{\infty}(\R^{d})$ as $t\to\infty$ (see Ben-Artzi, Koch and Saut 
\cite{BKS}), we expect that if $p>1+2/d$, then the (small) solution 
to (\ref{4NLS}) will scatter to the solution to the linearized equation 
and if $p\le1+2/d$, then the solution to (\ref{4NLS}) will not scatter. 
The homogeneous fourth order nonlinear Schr\"odinger type equation 
\begin{eqnarray}
i\pt_tu+\frac12\Delta^{2} u=\lambda|u|^{p-1}u,
\qquad t>0,x\in\rre^{d}
\label{h}
\end{eqnarray}
has been studied by many authors from the point of view of the scattering.  
See \cite{SS} for a review of the known results on the scattering and blow-up problem 
for (\ref{h}). Compared to the 
homogeneous equation (\ref{h}), there are few results on the long time behavior 
of solution for (\ref{4NLS}). 
For the one dimensional cubic case, the second author \cite{S} proved that 
for a given asymptotic profile, there exists a solution $u$ to (\ref{4NLS}) 
which converges to the given asymptotic profile as $t\to\infty$, 
where the asymptotic profile is given by the leading term of the solution 
to the linearized equation with a logarithmic phase correction. 
Furthermore, Hayashi and Naumkin \cite{HN9} proved that for any small 
initial data, there exists a global solution to (\ref{4NLS}) with $d=1,p=3$ 
which behaves like a solution to the linearized equation 
with a logarithmic phase correction. Recently, 
the authors \cite{SS} have shown the unique existence of solution $u$ to (\ref{4NLS}) 
which scatters to the free solution for $2<p<3$ if $d=2$ and $9/5<p<7/3$ if $d=3$. 
In this paper, refining the asymptotic formula \cite[Proposition 2.1]{SS} 
as $t\to\infty$ for the solution to the linearized equation of (\ref{4NLS}), we prove 
the long range scattering for (\ref{4NLS}) with the quadratic nonlinearity 
in two dimensions. 

Let us consider the final state problem:
\begin{eqnarray}
\left\{
\begin{array}{l}
\displaystyle{
i\pt_tu+\frac12\Delta u-\frac{1}{4}\pt_{x_{1}}^4u=\lambda|u|u,
\qquad t>0,x\in\rre^{2},}\\
\displaystyle{\lim_{t\to+\infty}(u(t)-u_{+}(t))=0,\qquad in\ L^2(\rre^{2}),}
\end{array}
\right.
\label{FSP}
\end{eqnarray}
where $u:\rre\times\rre^{2}\to\C$ is an unknown function and 
$u_{+}:\rre\times\rre^{2}\to\C$ is a ``modified'' asymptotic profile 
given by
\begin{equation}\label{u}
\begin{aligned}
\quad u_+(t,x)&=
\frac{t^{-1}}{\sqrt{3\mu_{1}^{2}+1}}\hat{\psi}_{+}(\mu)
e^{\frac34it\mu_{1}^{4}+\frac12it|\mu|^{2}
+iS_{+}(t,\mu)-i\frac{\pi}{2}},\\
S_+(t,\xi)&=-\lambda
\frac{|\hat{\psi}_{+}(\xi)|}{\sqrt{3\xi_{1}^{2}+1}}\log t,
\end{aligned}
\end{equation}
where $\mu=(\mu_{1},\mu_{2})\in\rre\times\rre$ is a stationary point 
for the oscillatory integral (\ref{os}) associated with the linearized equation 
of (\ref{FSP}), i.e.,
\begin{equation}\label{mu}
\begin{aligned}
\mu_{1}&=
\frac{1}{2^{\frac13}}\left\{
\left(\frac{x_{1}}{t}+\sqrt{
\left(\frac{x_{1}}{t}
\right)^{2}+\frac{4}{27}}\right)
\right\}^{1/3}
+
\left\{\left(\frac{x_{1}}{t}-\sqrt{
\left(\frac{x_{1}}{t}\right)^{2}+\frac{4}{27}}\right)
\right\}^{1/3},\\
\mu_{2}&=\frac{x_{2}}{t}.
\end{aligned}
\end{equation}

Our main result in this paper is as follows:

\begin{theorem}[Long range scattering]\label{nonlinear1}
There exists $\varepsilon>0$ with the following properties: 
for any $\psi_+\in H^{0,2}(\R^2)$ with $\|\psi_+\|_{H^{0,2}}<\varepsilon$ 
(see (\ref{ws}) for the definition of $H^{0,2}$), 
there exists a unique global solution $u\in C(\rre;L_x^2(\rre^{2}))
\cap \langle\pt_{x_{1}}\rangle^{-1/4}L_{loc}^{4}(\rre;L_x^{4}(\rre^{2}))$ 
to (\ref{FSP}) satisfying 
\begin{eqnarray*}
\|u(t)-u_{+}(t)\|_{L_x^2}
\le Ct^{-\alpha}
\end{eqnarray*}
for $t\ge3$, where $1/2<\alpha<3/4$ and 
$u_{+}$ is given by (\ref{u}). 
\end{theorem}

We give an outline of the proof of Theorem \ref{nonlinear1}. 
To prove Theorem \ref{nonlinear1}, we employ the argument 
by Ozawa \cite{O}, Hayashi and Naumkin \cite{HN3,HN1}. 
We first construct a solution $u$ to the final state problem 
\begin{eqnarray}
\left\{
\begin{array}{l}
\displaystyle{
i\pt_tu+\frac12\Delta u-\frac{1}{4}\pt_{x_{1}}^4u=\lambda|u|u,
\qquad t>0,x\in\rre^{2},}\\
\displaystyle{\lim_{t\to+\infty}(u(t)-W(t){{\mathcal F}}^{-1}w)=0,\qquad in\ L^2(\R^{2}).}
\end{array}
\right.
\label{FSP2}
\end{eqnarray}
where $w(t,\xi)=\hat{\psi}_+(\xi)e^{iS_+(t,\xi)}$ and $\{W(t)\}_{t\in\rre}$ 
is a unitary group generated by the operator $(1/2)i\Delta-(1/4)i\pt_{x_{1}}^4$. 
To prove this, 
we first rewrite (\ref{FSP2}) as the integral equation 
\begin{eqnarray}
\lefteqn{u(t)-
W(t){{{\mathcal F}}}^{-1}w}\nonumber\\
&=&i\lambda\int_t^{+\infty}
W(t-\tau)[|u|u-|W(t){{{{\mathcal F}}}}^{-1}w|
W(t){{{{\mathcal F}}}}^{-1}w](\tau)d\tau\nonumber\\
& &-i\int_t^{+\infty}W(t-\tau)R(\tau)d\tau,\label{MINT}
\end{eqnarray}
where
\begin{eqnarray*}
R
=W(t){{\mathcal F}}^{-1}\left[
\frac{\lambda t^{-1}}
{\sqrt{3\xi_{1}^{2}+1}}|\hat{\psi}_+|
\hat{\psi}_+(\xi)e^{iS_{+}(t,\xi)}\right]
-\lambda|W(t){{\mathcal F}}^{-1}w|W(t){{\mathcal F}}^{-1}w.
\end{eqnarray*}
Next, we apply the contraction mapping principle to 
the integral equation (\ref{MINT})  in a suitable function space. 
In this step the asymptotic formula (Proposition \ref{linear}) 
and the Strichartz estimate (Lemma \ref{S}) for 
the linear equation (\ref{4LS}) play an important role. 
Finally, we show that the solutions of (\ref{FSP2}) 
converge to $u_{+}$ in $L^{2}$ as $t\to\infty$. 

\begin{remark} 
It is not likely that our proof will be applicable for 
the three dimensional critical case (i.e., (\ref{4NLS}) 
with $d=3$ and $p=5/3$) due to the lack of smoothness 
of the nonlinear term.
\end{remark}


By using the argument by Glassey \cite{G} we can prove the non-existence 
of asymptotically free solution for (\ref{4NLS}) with $p\le1+2/d$. 

\begin{theorem}[Nonexistence of asymptotically free solution]\label{nonlinear3}
Let $d\ge2$ and $1<p\le1+2/d$. 
Let $u\in C(\rre,L^{2}(\rre^{d}))$ be a solution to (\ref{4NLS}) with 
$u(0,x)=u_{0}\in L^{2}(\rre^{d})$. Assume that there exists a function 
$\psi_{+}\in H^{0,s}\cap\langle\pt_{x_{1}}\rangle^{p-1}L^{2/p}$  
with $s>(4-d)/2+(d-1)p/2$ 
such that
\begin{eqnarray}
\|u(t)-W(t)\psi_{+}\|_{L_{x}^{2}}\to0,\label{asym}
\end{eqnarray} 
as $t\to\infty$, where $\{W(t)\}_{t\in\rre}$ 
is a unitary group generated by the operator $(1/2)i\Delta-(1/4)i\pt_{x_{1}}^4$. 
Then $u\equiv0$.
\end{theorem}


We introduce several notations and function spaces 
which are used throughout this paper. 
For $\psi\in{{\mathcal S}}'(\rre^{d})$, 
$\hat{\psi}(\xi)={{\mathcal F}}[\psi](\xi)$ 
denote the Fourier transform of $\psi$. Let 
$\langle\xi\rangle=\sqrt{|\xi|^2+1}$. The differential operator
$\langle\nabla\rangle^s=(1-\Delta)^{s/2}$ denotes  
the Bessel potential of order $-s$. 
We define $\langle\pt_{x_{1}}\rangle^{s}={{\mathcal F}}^{-1}
\langle\xi_{1}\rangle^{s}{{\mathcal F}}^{-1}$ for $s\in\rre$. 
For $1\le q,r\le\infty$, $L^q(t,\infty;L_x^r(\rre^{d}))$ is defined 
as follows:
\begin{eqnarray*}
L^q(t,\infty;L_x^r(\rre^{d}))&=&\{u\in{{\mathcal S}}'(\rre^{1+d});
\|u\|_{L^q(t,\infty;L_x^r)}<\infty\},\\
\|u\|_{L^q(t,\infty;L_x^r)}&=&
\left(\int_t^{\infty}\|u(\tau)\|_{L_x^r}^qd\tau\right)^{1/q}.
\end{eqnarray*}
We will use the  Sobolev spaces
\begin{eqnarray*}
H^s(\R^d)=\lbrace \phi \in \mathcal S'(\R^d); 
\|\phi\|_{H^s}=\| 
\langle \nabla\rangle^{s} \phi\|_{L^2}<\infty\rbrace
\end{eqnarray*}
and the weighted Sobolev spaces 
\begin{eqnarray}
H^{m,s}(\R^d)=\lbrace \phi \in \mathcal S'(\R^d); 
\|\phi\|_{H^{m,s}}=\|\langle x\rangle^{s} 
\langle \nabla\rangle^{m} \phi\|_{L^2}<\infty\rbrace.
\label{ws}
\end{eqnarray}
We denote 
various constants by $C$ and so forth. They may differ from 
line to line, when this does not cause any confusion. 

The plan of the present paper is as follows. 
In Section 2, we prove several linear estimates for the fourth order 
Schr\"{o}dinger type equation (\ref{4LS}). In Section 3, 
we prove Theorem 
\ref{nonlinear1} by applying the contraction mapping principle to
the integral equation (\ref{MINT}). 
Finally in Section 4, we give the proof of Theorem 
\ref{nonlinear3}.

\section{Linear Estimates} \label{sec:linear}

In this section, we derive several linear estimates 
that will be crucial for the proof of Theorem 1.1, for  
the fourth order Schr\"{o}dinger type equation
\begin{eqnarray}
\left\{
\begin{array}{l}
\displaystyle{
i\pt_tu+\frac12\Delta u-\frac{1}{4}\pt_{x_{1}}^4u=0,
\qquad t>0,x\in\rre^{2},
}\\
\displaystyle{u(0,x)=\psi(x),\qquad\qquad\qquad\ \  x\in\rre^{2}.}
\end{array}
\right.
\label{4LS}
\end{eqnarray}
The solution to (\ref{4LS}) can be rewritten as
\begin{eqnarray}
u(t,x)=[W(t)\psi](x)
=\frac{1}{2\pi}
\int_{\rre^{2}}e^{ix\xi-\frac{i}{2}t|\xi|^{2}-\frac{i}{4}t\xi_{1}^{4}}
\hat{\psi}(\xi)d\xi.\label{os}
\end{eqnarray}

The following proposition is a 
refinement of \cite[Proposition 2.1.]{SS} for $d=2$ 
and $p=2$. 

\begin{proposition}\label{linear} 
We have
\begin{eqnarray*}
[W(t)\psi](x)=
\frac{t^{-1}}{\sqrt{3\mu_{1}^{2}+1}}\hat{\psi}(\mu)
e^{\frac34it\mu_{1}^{4}+\frac12it|\mu|^{2}
-i\frac{\pi}{2}}+R(t,x)
\end{eqnarray*}
for $t\ge2$,
where $\mu=(\mu_{1},\mu_{2})$ is given by (\ref{mu}) 
and $R$ satisfies
\begin{eqnarray*}
\|R(t)\|_{L_x^{2}}
\le Ct^{-\beta}\|\psi\|_{H_{x}^{0,2}},
\end{eqnarray*}
for $0<\beta<3/4$.
\end{proposition}
 
\begin{proof}[Proof of Proposition \ref{linear}.]
We easily see 
\begin{eqnarray*}
[W(t)\psi](x)
=
\int_{\rre^{2}}K(t,x-y)\psi(y)dy,
\end{eqnarray*}
where
\begin{eqnarray*}
K(t,z)
=\left(\frac{1}{2\pi}\right)^{2}
\int_{\rre^{2}}e^{iz\xi-\frac{i}{2}t|\xi|^{2}-\frac{i}{4}t\xi_{1}^{4}}
d\xi.
\end{eqnarray*}
By the Fresnel integral formula
\begin{eqnarray*}
\frac{1}{\sqrt{2\pi}}
\int_{\rre}e^{iz_{2}\xi_{2}-\frac{i}{2}t\xi_{2}^{2}}d\xi_{2}
=t^{-\frac12}e^{\frac{iz_{2}^{2}}{2t}-i\frac{\pi}{4}},
\end{eqnarray*}
we have
\begin{eqnarray*}
K(t,z)
=\left(\frac{1}{2\pi}\right)^{\frac{3}{2}}
t^{-\frac{1}{2}}e^{\frac{iz_{2}^{2}}{2t}-i\frac{\pi}{4}}
\int_{\rre}e^{iz_{1}\xi_{1}-\frac{i}{2}t\xi_{1}^{2}-\frac{i}{4}t\xi_{1}^{4}}
d\xi_{1}.
\end{eqnarray*}
Therefore, we find
\begin{eqnarray*}
u(t,x)
=\frac{1}{\sqrt{2\pi}}
t^{-\frac{1}{2}}e^{\frac{i}{2}t\mu_{2}^{2}-i\frac{\pi}{4}}
\int_{\rre}e^{ix_{1}\xi_{1}-\frac{i}{2}t\xi_{1}^{2}-\frac{i}{4}t\xi_{1}^{4}}
{{\mathcal F}}[e^{\frac{iy_{2}^{2}}{2t}}\psi](\xi_{1},\mu_{2})d\xi_{1}.
\end{eqnarray*}
We split $u$ into the following two pieces: 
\begin{eqnarray}
\lefteqn{u(t,x)}\nonumber\\
&=&\frac{1}{\sqrt{2\pi}}
t^{-\frac{1}{2}}e^{\frac{i}{2}t\mu_{2}^{2}-i\frac{1}{4}\pi}
\int_{\rre}e^{ix_{1}\xi_{1}-\frac{i}{2}t\xi_{1}^{2}-\frac{i}{4}t\xi_{1}^{4}}
{{\mathcal F}}[\psi](\xi_{1},\mu_{2})d\xi_{1}
\nonumber\\
& &+\frac{1}{\sqrt{2\pi}}
t^{-\frac{1}{2}}e^{\frac{i}{2}t\mu_{2}^{2}-i\frac{1}{4}\pi}
\int_{\rre}e^{ix_{1}\xi_{1}-\frac{i}{2}t\xi_{1}^{2}-\frac{i}{4}t\xi_{1}^{4}}
{{\mathcal F}}[(e^{\frac{iy_{2}^{2}}{2t}}-1)\psi](\xi_{1},\mu_{2})d\xi_{1}
\nonumber\\
&=:&L(t,x)+R(t,x).
\label{u1}
\end{eqnarray}
To evaluate $L$, we split $L$ into 
\begin{eqnarray}
L(t,x)&=&
\frac{1}{\sqrt{2\pi}}
t^{-\frac{1}{2}}e^{\frac{i}{2}t\mu_{2}^{2}-i\frac{1}{4}\pi}
{{\mathcal F}}[\psi](\mu)
\int_{\rre}e^{ix_{1}\xi_{1}-\frac{i}{2}t\xi_{1}^{2}-\frac{i}{4}t\xi_{1}^{4}}
d\xi_{1}\nonumber\\
& &+
\frac{1}{\sqrt{2\pi}}
t^{-\frac{1}{2}}e^{\frac{i}{2}t\mu_{2}^{2}-i\frac{1}{4}\pi}
\pt_{\xi_{1}}{{\mathcal F}}[\psi](\mu)
\int_{\rre}e^{ix_{1}\xi_{1}-\frac{i}{2}t\xi_{1}^{2}-\frac{i}{4}t\xi_{1}^{4}}
(\xi_{1}-\mu_{1})d\xi_{1}\nonumber\\
& &+\frac{1}{\sqrt{2\pi}}
t^{-\frac{1}{2}}e^{\frac{i}{2}t\mu_{2}^{2}-i\frac{1}{4}\pi}
\int_{\rre}e^{ix_{1}\xi_{1}-\frac{i}{2}t\xi_{1}^{2}-\frac{i}{4}t\xi_{1}^{4}}
\nonumber\\
& &\qquad\quad\times
({{\mathcal F}}[\psi](\xi_{1},\mu_{2})
-{{\mathcal F}}[\psi](\mu_{1},\mu_{2})
-\pt_{\xi_{1}}{{\mathcal F}}[\psi](\mu)(\xi_{1}-\mu_{1}))d\xi_{1}\nonumber\\
&=:&L_{1}(t,x)+L_{2}(t,x)+L_{3}(t,x).
\label{u2}
\end{eqnarray}
We rewrite $L_{1}$ as follows:
\begin{eqnarray*}
L_{1}(t,x)=
\frac{1}{\sqrt{2\pi}}
t^{-\frac{1}{2}}
e^{\frac34it\mu_{1}^{4}+\frac{i}{2}t|\mu|^{2}-i\frac{1}{4}\pi}
{{\mathcal F}}[\psi](\mu)
\int_{\rre}e^{-itS(\mu_{1},\xi_{1})}d\xi_{1},
\end{eqnarray*}
where 
$S(\mu_{1},\xi_{1})$ is defined by 
\begin{eqnarray*}
S(\mu_{1},\xi_{1})
=\frac{1}{4}\xi_{1}^{4}+\frac{1}{2}\xi_{1}^{2}
-(\mu_{1}^{3}+\mu_{1})\xi_{1}+\frac34\mu_{1}^{4}+\frac12\mu_{1}^{2}.
\end{eqnarray*}
Let
\begin{eqnarray*}
\eta_{1}
=\mu_{1}+\frac{1}{\sqrt{2}}\frac{1}{\sqrt{3\mu_{1}^2+1}}
(\xi_{1}-\mu_{1})
\sqrt{\xi_{1}^2+2\mu_{1}\xi_{1}+3\mu_{1}^2+2}.
\end{eqnarray*}
Then, $L_{1}$ can be rewritten as follows:
\begin{eqnarray}
\lefteqn{L_{1}(t,x)}\nonumber\\
&=&
\frac{1}{\sqrt{2\pi}}
t^{-\frac{1}{2}}
{{\mathcal F}}[\psi](\mu)
e^{\frac34it\mu_{1}^{4}+\frac{i}{2}t|\mu|^{2}-i\frac{1}{4}\pi}
\nonumber\\
& &\times
\left\{\int_{\rre}
e^{-itS(\mu_{1},\xi_{1})}
\left.\frac{d\eta_{1}}{d\xi_{1}}d\xi_{1}
-\frac{d^{2}\eta_{1}}{d\xi_{1}^{2}}\right|_{\xi_{1}=\mu_{1}}\int_{\rre}
e^{-itS(\mu_{1}\xi_{1})}
(\xi_{1}-\mu_{1})d\xi_{1}
\right.\nonumber\\
& &\quad+
\left.
\int_{\rre}
e^{-itS(\mu_{1}\xi_{1})}
(1-\frac{d\eta_{1}}{d\xi_{1}}+
\left.\frac{d^{2}\eta_{1}}{d\xi_{1}^{2}}\right|_{\xi_{1}=\mu_{1}}(\xi_{1}-\mu_{1}))d\xi_{1}
\right\}\nonumber\\
&=:&
t^{-\frac{1}{2}}
{{\mathcal F}}[\psi](\mu)
e^{\frac34it\mu_{1}^{4}+\frac{i}{2}t|\mu|^{2}-i\frac{1}{4}\pi}
(L_{1,1}(t,x)+L_{1,2}(t,x)+L_{1,3}(t,x)).
\nonumber\\
\label{u3}
\end{eqnarray}
For $L_{1,1}$, changing the variable $\xi_{1}\mapsto\eta_{1}$, 
we have
\begin{eqnarray*}
L_{1,1}(t,x)=\frac{1}{\sqrt{2\pi}}
\int_{\rre}
e^{-\frac12it(3\mu_{1}^2+1)(\eta_{1}-\mu_{1})^2}d\eta_{1}.
\end{eqnarray*}
In addition, changing the variable $\zeta_{1}=(1/\sqrt{2})t^{1/2}
\sqrt{3\mu_{1}^2+1}(\eta_{1}-\mu_{1})$ ($\eta_{1}\mapsto\zeta_{1}$) and 
using the Fresnel integral formula, we obtain
\begin{eqnarray}
L_{1,1}(t,x)=\sqrt{\frac{2}{\pi}}
\frac{t^{-1/2}}{\sqrt{3\mu_{1}^{2}+1}}
\int_{\rre}e^{-i\zeta^2}d\zeta
=
\frac{t^{-1/2}}{\sqrt{3\mu_{1}^{2}+1}}e^{-i\frac{\pi}{4}}.
\label{u4}
\end{eqnarray}
Next we evaluate $L_{1,2}$. Integrating by parts via the identity
\begin{eqnarray}
e^{-itS(\mu_{1},\xi_{1})}(\xi_{1}-\mu_{1})
=it^{-1}G(\mu_{1},\xi_{1})\pt_{\xi_{1}}e^{-itS(\mu_{1},\xi_{1})}
\label{ip1}
\end{eqnarray}
with 
\begin{eqnarray*}
G(\mu_{1},\xi_{1})=\frac{1}{\xi_{1}^{2}+\mu_{1}\xi_{1}+\mu_{1}^{2}+1},
\end{eqnarray*}
we have
\begin{eqnarray*}
\int_{\rre}e^{-itS(\mu_{1},\xi_{1})}(\xi_{1}-\mu_{1})d\xi_{1}
=-it^{-1}\int_{\rre}e^{-it
S(\mu_{1},\xi_{1})}\pt_{\xi_{1}}G(\mu_{1},\xi_{1})
d\xi_{1}.
\end{eqnarray*}
Furthermore, integrating by parts via the identity
\begin{eqnarray}
e^{-itS(\mu_{1},\xi_{1})}
=H(t,\mu_{1},\xi_{1})\pt_{\xi_{1}}\{(\xi_{1}-\mu_{1})e^{-itS(\mu_{1},\xi_{1})}\}
\label{ip2}
\end{eqnarray}
with 
\begin{eqnarray*}
H(t,\mu_{1},\xi_{1})=\frac{1}
{1-it(\xi_{1}-\mu_{1})^{2}(\xi_{1}^{2}+\mu_{1}\xi_{1}+\mu_{1}^{2}+1)},
\end{eqnarray*}
we obtain
\begin{eqnarray*}
\lefteqn{\int_{\rre}e^{-itS(\mu_{1},\xi_{1})}(\xi_{1}-\mu_{1})d\xi_{1}}\\
&=&it^{-1}\int_{\rre}e^{-itS(\mu_{1},\xi_{1})}(\xi_{1}-\mu_{1})
\pt_{\xi_{1}}\left\{
H(t,\mu_{1},\xi_{1})
\pt_{\xi_{1}}G(\mu_{1},\xi_{1})
\right\}d\xi_{1}\\
&=&it^{-1}\int_{\rre}e^{-itS(\mu_{1},\xi_{1})}(\xi_{1}-\mu_{1})
H(t,\mu_{1},\xi_{1})
\pt_{\xi_{1}}^{2}G(\mu_{1},\xi_{1})
d\xi_{1}\\
& &+it^{-1}\int_{\rre}e^{-itS(\mu_{1},\xi_{1})}(\xi_{1}-\mu_{1})
\pt_{\xi_{1}}H(t,\mu_{1},\xi_{1})
\pt_{\xi_{1}}G(\mu_{1},\xi_{1})
d\xi_{1}.
\end{eqnarray*}
Using the inequalities
\begin{eqnarray}
\left|\pt_{\xi_{1}}^{j}G(\mu_{1},\xi_{1})
\right|&\le& C\langle\mu_{1}\rangle^{-j-2},\label{seo}\\
\left|\pt_{\xi_{1}}^{j}H(t,\mu_{1},\xi_{1})
\right|&\le& C
\frac{|\xi_{1}-\mu_{1}|^{-j}}{1+t(\xi_{1}-\mu_{1})^{2}(\xi_{1}^{2}+\mu_{1}\xi_{1}+\mu_{1}^{2}+1)}\label{se}
\end{eqnarray}
for $j=0,1,2$, we have 
\begin{eqnarray*}
\lefteqn{\left|
\int_{\rre}e^{-itS(\mu_{1},\xi_{1})}(\xi_{1}-\mu_{1})d\xi_{1}
\right|}\\
&\le&Ct^{-1}\langle\mu_{1}\rangle^{-4}\int_{\rre}
\frac{|\xi_{1}-\mu_{1}|+\langle\mu_{1}\rangle
}{1+t(\xi_{1}-\mu_{1})^{2}(\xi_{1}^{2}+\mu_{1}\xi_{1}+\mu_{1}^{2}+1)}
d\xi_{1}.
\end{eqnarray*}
By using the inequalities 
\begin{eqnarray*}
\xi_{1}^2+\mu_{1}\xi_{1}+\mu_{1}^2+1\ge
\left\{
\begin{array}{l}
\medskip\displaystyle{
\frac12\langle\mu_{1}\rangle^{2},\ \qquad\text{if}
\ |\xi_{1}-\mu_{1}|\le\langle\mu_{1}\rangle,}\\
\displaystyle{
\frac14(\xi_{1}-\mu_{1})^2,\qquad\text{if}
\ |\xi_{1}-\mu_{1}|\ge\langle\mu_{1}\rangle,}
\end{array}
\right.
\end{eqnarray*}
we see 
\begin{eqnarray}
\lefteqn{\left|\int_{\rre}e^{-itS(\mu_{1},\xi_{1})}(\xi_{1}-\mu_{1})d\xi_{1}
\right|}\nonumber\\
&\le&Ct^{-1}\langle\mu_{1}\rangle^{-3}\int_{|\xi_{1}-\mu_{1}|\le\langle\mu_{1}\rangle}
\frac{1}{
1+t\langle\mu_{1}\rangle^{2}(\xi_{1}-\mu_{1})^2}d\xi_{1}
\nonumber\\
& &+Ct^{-1}\langle\mu_{1}\rangle^{-4}
\int_{|\xi_{1}-\mu_{1}|\ge\langle\mu_{1}\rangle}
\frac{|\xi_{1}-\mu_{1}|}{1+t(\xi_{1}-\mu_{1})^4}d\xi_{1}\nonumber\\
&\le&Ct^{-1-\gamma}\langle\mu_{1}\rangle^{-2\gamma-3}
\int_{|\xi_{1}-\mu_{1}|\le\langle\mu_{1}\rangle}
|\xi_{1}-\mu_{1}|^{-2\gamma}d\xi_{1}
\nonumber\\
& &+Ct^{-2}\langle\mu_{1}\rangle^{-4}
\int_{|\xi_{1}-\mu_{1}|\ge\langle\mu_{1}\rangle}
|\xi_{1}-\mu_{1}|^{-3}d\xi_{1}\nonumber\\
&\le&C(t^{-1-\gamma}\langle\mu_{1}\rangle^{-4\gamma-2}
+t^{-2}\langle\mu_{1}\rangle^{-6})\nonumber\\
&\le&Ct^{-1-\gamma}\langle\mu_{1}\rangle^{-4\gamma-2},
\label{short}
\end{eqnarray}
where $0<\gamma<1/2$. 
Hence
\begin{eqnarray}
|L_{1,2}(t,x)|
\le
Ct^{-1-\gamma}\left|
\left.\frac{d^{2}\eta_{1}}{d\xi^{2}}
\right|_{\xi_{1}=\mu_{1}}\right|\langle\mu_{1}\rangle^{-4\gamma-2}
\le Ct^{-1-\gamma}\langle\mu_{1}\rangle^{-4\gamma-3}.
\label{short3}
\end{eqnarray}

For $L_{1,3}$, integrating by parts via the identity (\ref{ip2}), we have
\begin{eqnarray*}
L_{1,3}(t,x)
&=&-\frac{1}{\sqrt{2\pi}}\int_{\rre}(\xi_{1}-\mu_{1})e^{-itS(\mu_{1},\xi_{1})}\\
& &
\qquad\times
\pt_{\xi_{1}}\left\{
H(t,\mu_{1},\xi_{1})(1-\frac{d\eta_{1}}{d\xi_{1}}+
\left.\frac{d^{2}\eta_{1}}{d\xi_{1}^{2}}\right|_{\xi_{1}=\mu_{1}}(\xi_{1}-\mu_{1}))\right\}d\xi_{1}.
\nonumber
\end{eqnarray*}
Furthermore, integrating by parts via the identity (\ref{ip1}), 
we have
\begin{eqnarray}
\lefteqn{L_{1,3}(t,x)}
\nonumber\\
&=&\frac{it^{-1}}{\sqrt{2\pi}}\int_{\rre}e^{-itS(\mu_{1},\xi_{1})}
\nonumber\\
& &\quad\times
\pt_{\xi_{1}}\left[G(\mu_{1},\xi_{1})
\pt_{\xi_{1}}\left\{H(t,\mu_{1},\xi_{1})
(1-\frac{d\eta_{1}}{d\xi_{1}}+\left.\frac{d^{2}\eta_{1}}{
d\xi_{1}^{2}}\right|_{\xi_{1}=\mu_{1}}(\xi_{1}-\mu_{1}))\right\}\right]d\xi_{1}\nonumber\\
&=&\frac{it^{-1}}{\sqrt{2\pi}}\int_{\rre}e^{-itS(\mu_{1},\xi_{1})}
F_{1}(\mu_{1},\xi_{1})\frac{d^{3}\eta_{1}}{d\xi_{1}^{3}}d\xi_{1}\nonumber\\
& &+\frac{it^{-1}}{\sqrt{2\pi}}\int_{\rre}e^{-itS(\mu_{1},\xi_{1})}
F_{2}(\mu_{1},\xi_{1})\left(\frac{d^{2}\eta_{1}}{d\xi_{1}^{2}}-\left.\frac{d^{2}\eta_{1}}{
d\xi_{1}^{2}}\right|_{\xi_{1}=\mu_{1}}\right)d\xi_{1}\nonumber\\
& &+\frac{it^{-1}}{\sqrt{2\pi}}\int_{\rre}e^{-itS(\mu_{1},\xi_{1})}
F_{3}(\mu_{1},\xi_{1})\left(1-\frac{d\eta_{1}}{d\xi_{1}}+\left.\frac{d^{2}\eta_{1}}{
d\xi_{1}^{2}}\right|_{\xi_{1}=\mu_{1}}(\xi_{1}-\mu_{1})\right)d\xi_{1},
\nonumber\\
\label{long}
\end{eqnarray}
where
\begin{eqnarray*}
F_{1}(\mu_{1},\xi_{1})
&=&
-G(\mu_{1},\xi_{1})H(t,\mu_{1},\xi_{1}),\\
F_{2}(\mu_{1},\xi_{1})
&=&-2
G(\mu_{1},\xi_{1})
\pt_{\xi_{1}}H(t,\mu_{1},\xi_{1})
-\pt_{\xi_{1}}G(\mu_{1},\xi_{1})H(t,\mu_{1},\xi_{1}),
\nonumber\\
F_{3}(\mu_{1},\xi_{1})
&=&G(\mu_{1},\xi_{1})
\pt_{\xi_{1}}^{2}
H(t,\mu_{1},\xi_{1})
+\pt_{\xi_{1}}G(\mu_{1},\xi_{1})
\pt_{\xi_{1}}H(t,\mu_{1},\xi_{1}).
\end{eqnarray*}
Since $\displaystyle{\left|\frac{d^3\eta_{1}}{d\xi_{1}^3}\right|
\le C\langle\mu_{1}\rangle^{-2}}$,
we see that
\begin{eqnarray*}
\left|\frac{d^2\eta_{1}}{d\xi_{1}^2}-
\left.\frac{d^2\eta_{1}}{d\xi_{1}^2}\right|_{\xi_{1}=\mu_{1}}\right|
&\le&\sup_{\xi_{1}\in\rre}\left|\frac{d^3\eta_{1}}{d\xi_{1}^3}\right|
|\xi_{1}-\mu_{1}|
\le C\langle\mu_{1}\rangle^{-2}|\xi_{1}-\mu_{1}|,
\end{eqnarray*}
\begin{eqnarray*}
\lefteqn{\left|1-\frac{d\eta_{1}}{d\xi_{1}}+
\left.\frac{d^2\eta_{1}}{d\xi_{1}^2}\right|_{\xi_{1}=\mu_{1}}(\xi_{1}-\mu_{1})\right|}\\
&=&\left|\left.\frac{d\eta_{1}}{d\xi_{1}}\right|_{\xi_{1}=\mu_{1}}
-\frac{d\eta_{1}}{d\xi_{1}}+
\left.\frac{d^2\eta_{1}}{d\xi_{1}^2}\right|_{\xi_{1}=\mu_{1}}(\xi_{1}-\mu_{1})\right|\\
&\le&\frac12\sup_{\xi_{1}\in\rre}\left|\frac{d^3\eta_{1}}{d\xi_{1}^3}\right|
|\xi_{1}-\mu_{1}|^{2}\le C\langle\mu_{1}\rangle^{-2}|\xi_{1}-\mu_{1}|^{2}.
\end{eqnarray*}
Combining (\ref{seo}) and (\ref{se}) with the above three  
inequalities, we have
\begin{eqnarray*}
|L_{1,3}(t,x)|
\le Ct^{-1}\langle\mu_{1}\rangle^{-5}\int_{\rre}
\frac{|\xi_{1}-\mu_{1}|+\langle\mu_{1}\rangle
}{1+t(\xi_{1}-\mu_{1})^{2}(\xi_{1}^{2}+\mu_{1}\xi_{1}+\mu_{1}^{2}+1)}
d\xi_{1}.
\end{eqnarray*}
Hence, by an argument similar to (\ref{short}), we have
\begin{eqnarray}
|L_{1,3}(t,x)|
\le Ct^{-1-\gamma}\langle\mu_{1}\rangle^{-4\gamma-3},
\label{u5}
\end{eqnarray}
where $0<\gamma<1/2$.
By (\ref{u3}), (\ref{u4}), (\ref{short3}) and (\ref{u5}), we have
\begin{eqnarray}
L_{1}(t,x)
=\frac{t^{-1}}{\sqrt{3\mu_{1}^{2}+1}}{{\mathcal F}}[\psi](\mu)
e^{\frac34it\mu_{1}^{4}+\frac{i}{2}t|\mu|^{2}-i\frac{\pi}{2}}
+R_{1}(t,x),\label{l}
\end{eqnarray}
where $R_{1}$ satisfies
\begin{eqnarray*}
|R_{1}(t,x)|\le 
Ct^{-\frac{3}{2}-\gamma}\langle\mu_{1}\rangle^{-4\gamma-3}
|{{\mathcal F}}[\psi](\mu)|
\end{eqnarray*}
with $0<\gamma<1/2$. Hence the Plancherel identity yield 
\begin{eqnarray}
\|R_{1}(t)\|_{L_{x}^{2}}
&\le&
Ct^{-\frac{1}{2}-\gamma}
\|\langle\mu_{1}\rangle^{-4\gamma-2}
{{\mathcal F}}[\psi](\mu)
\|_{L_{\mu}^{2}}\nonumber\\
&\le&
Ct^{-\frac{1}{2}-\gamma}
\|{{\mathcal F}}[\psi](\mu)
\|_{L_{\mu}^{2}}\nonumber\\
&=&Ct^{-\frac{1}{2}-\gamma}
\|\psi\|_{L_{x}^{2}}.\label{r11}
\end{eqnarray}

Next we evaluate $L_{2}$. 
By (\ref{short}), we obtain
\begin{eqnarray*}
|L_{2}(t,x)|
&\le&Ct^{-\frac{3}{2}-\gamma}\langle\mu_{1}\rangle^{-4\gamma-2}
|\pt_{\xi_{1}}{{\mathcal F}}[\psi](\mu)|.
\end{eqnarray*}
By an argument similar to that in (\ref{r11}), we have
\begin{eqnarray}
\|L_{2}(t)\|_{L_{x}^{2}}
\le Ct^{-\frac12-\gamma}
\|\psi\|_{H^{0,1}}.\label{u7}
\end{eqnarray}
Next, we evaluate $L_{3}$. We write 
\begin{eqnarray*}
L_{3}(t,x)=:
t^{-\frac{1}{2}}
e^{\frac34it\mu_{1}^{4}+\frac{i}{2}t|\mu|^{2}-i\frac{1}{4}\pi}
\tilde{L}_{3}(t,x).
\end{eqnarray*}
The same argument as that in (\ref{long}) yields that 
$\tilde{L}_{3}$ is equal to the right hand side of (\ref{long}) 
by replacing $1-\frac{d\eta_{1}}{d\xi_{1}}+
\left.\frac{d^{2}\eta_{1}}{d\xi_{1}^{2}}\right|_{\xi_{1}=\mu_{1}}(\xi_{1}-\mu_{1})$ 
by ${{\mathcal F}}[\psi](\xi_{1},\mu_{2})
-{{\mathcal F}}[\psi](\mu_{1},\mu_{2})
-\pt_{\xi_{1}}{{\mathcal F}}[\psi](\mu)(\xi_{1}-\mu_{1})
$. Since
\begin{eqnarray*}
\left|
\pt_{\xi_{1}}{{\mathcal F}}[\psi](\xi_{1},\mu_{2})-\pt_{\xi_{1}}{{\mathcal F}}[\psi](\mu)
\right|
&\le&\|\pt_{\xi_{1}}^{2}{{\mathcal F}}[\psi](\cdot,\mu_{2})
\|_{L_{\xi_{1}}^{2}}|\xi_{1}-\mu_{1}|^{\frac12},
\end{eqnarray*}
\begin{eqnarray*}
\lefteqn{\left|{{\mathcal F}}[\psi](\xi_{1},\mu_{2})
-{{\mathcal F}}[\psi](\mu_{1},\mu_{2})
-\pt_{\xi_{1}}{{\mathcal F}}[\psi](\mu)(\xi_{1}-\mu_{1})
\right|}\qquad\qquad\qquad\qquad\qquad\qquad\\
&\le&\frac23\|\pt_{\xi_{1}}^{2}{{\mathcal F}}[\psi](\cdot,\mu_{2})\|_{L_{\xi_{1}}^{2}}
|\xi_{1}-\mu_{1}|^{\frac32},
\end{eqnarray*}
we have
\begin{eqnarray*}
\lefteqn{|\tilde{L}_{3}(t,x)|}\\
&\le& 
Ct^{-1}\|\pt_{\xi_{1}}^{2}{{\mathcal F}}[\psi]
(\cdot,\mu_{2})\|_{L_{\xi_{1}}^{2}}\\
& &
\qquad\times\left\{\langle\mu_{1}\rangle^{-2}
\left(\int_{\rre}
\frac{1}{
\{1+t(\xi_{1}-\mu_{1})^{2}(\xi_{1}^2+\mu_{1}\xi_{1}+\mu_{1}^2+1)\}^{2}}d\xi_{1}
\right)^{\frac12}\right.\\
& &\qquad\ \ \ \left.+
\langle\mu_{1}\rangle^{-3}
\int_{\rre}
\frac{|\xi_{1}-\mu_{1}|^{\frac12}+\langle\mu_{1}\rangle|\xi_{1}-\mu_{1}|^{-\frac12}}{
1+t(\xi_{1}-\mu_{1})^{2}(\xi_{1}^2+\mu_{1}\xi_{1}+\mu_{1}^2+1)}d\xi_{1}\right\}.
\end{eqnarray*}
By an argument similar to that in (\ref{short}), we obtain 
\begin{eqnarray*}
|\tilde{L}_{3}(t,x)|
\le Ct^{-1-\gamma}\langle\mu_{1}\rangle^{-4\gamma-\frac32}
\|\pt_{\xi_{1}}^{2}{{\mathcal F}}[\psi]
(\cdot,\mu_{2})\|_{L_{\xi_{1}}^{2}},
\end{eqnarray*}
where $0<\gamma<1/4$. Hence
\begin{eqnarray*}
\|L_{3}(t)\|_{L_{x_{1}}^{2}}
&\le& Ct^{-\frac{3}{2}-\gamma}
\|\langle\mu_{1}\rangle^{-4\gamma-\frac32}\|_{L_{x_{1}}^{2}}
\|\pt_{\xi_{1}}^{2}{{\mathcal F}}[\psi](\cdot,\mu_{2})
\|_{L_{\xi_{1}}^{2}}\\
&\le& Ct^{-1-\gamma}
\|\langle\mu_{1}\rangle^{-4\gamma-\frac12}\|_{L_{\mu_{1}}^{2}}
\|\pt_{\xi_{1}}^{2}{{\mathcal F}}[\psi](\cdot,\mu_{2})
\|_{L_{\xi_{1}}^{2}}\\
&\le& Ct^{-1-\gamma}
\|\pt_{\xi_{1}}^{2}{{\mathcal F}}[\psi](\cdot,\mu_{2}).
\|_{L_{\xi_{1}}^{2}}
\end{eqnarray*}
Combining the above inequality and the Plancherel identity, 
we have 
\begin{eqnarray}
\|L_{3}(t)\|_{L_{x}^{2}}
&\le&Ct^{-1-\gamma}
\|\|\pt_{\xi_{1}}^{2}{{\mathcal F}}[\psi](\cdot,\mu_{2})
\|_{L_{\mu_{1}}^{2}}\|_{L_{x_{2}}^{2}}
\nonumber\\
&\le&Ct^{-\frac12-\gamma}
\|\psi\|_{H^{0,2}}.\label{r12}
\end{eqnarray}
Finally let us evaluate $R$. $R$ can be rewritten as 
\begin{eqnarray*}
R=
t^{-\frac{1}{2}}e^{\frac{i}{2}t|\mu_{2}|^{2}-i\frac{1}{4}\pi}
W_{4LS}(t){{\mathcal F}}_{x_{2}\mapsto\xi_{2}}
[(e^{\frac{iy_{2}^{2}}{2t}}-1)\psi](x_{1},\mu_{2}),
\end{eqnarray*}
where $\{W_{4LS}(t)\}_{t\in\rre}$ is a unitary group generated by 
the linear operator $(i/2)\pt_{x_{1}}^{2}
-(i/4)\pt_{x_{1}}^{4}$:
\begin{eqnarray*}
W_{4LS}(t)\phi=\frac{1}{\sqrt{2\pi}}
\int_{\rre}e^{ix_{1}\xi_{1}-\frac{i}{2}t\xi_{1}^{2}-\frac{i}{4}t\xi_{1}^{4}}
{{\mathcal F}}_{x_{1}\mapsto\xi_{1}}[\phi](\xi_{1})d\xi_{1}.
\end{eqnarray*}
Then, we obtain
\begin{eqnarray*}
\|R(t)\|_{L_{x_{1}}^{2}}
&=&t^{-\frac{1}{2}}\|
{{\mathcal F}}_{x_{2}\mapsto\xi_{2}}[(e^{\frac{iy_{2}^{2}}{2t}}-1)
\psi](x_{1},\mu_{2})\|_{L_{x_{1}}^{2}}.
\end{eqnarray*}
Combining the above identity and the Plancherel identity, 
we have
\begin{eqnarray}
\|R(t)\|_{L_{x}^{2}}
&=&t^{-\frac{1}{2}}\|\|
{{\mathcal F}}_{x_{2}\mapsto\xi_{2}}[(e^{\frac{iy_{2}^{2}}{2t}}-1)
\psi](x_{1},\mu_{2})\|_{L_{x_{1}}^{2}}\|_{L_{x_{2}}^{2}}
\nonumber\\
&=&t^{-\frac{1}{2}}\|\|
{{\mathcal F}}_{x_{2}\mapsto\xi_{2}}[(e^{\frac{iy_{2}^{2}}{2t}}-1)
\psi](x_{1},\mu_{2})\|_{L_{x_{2}}^{2}}\|_{L_{x_{1}}^{2}}
\nonumber\\
&=&\|\|
{{\mathcal F}}_{x_{2}\mapsto\xi_{2}}[(e^{\frac{iy_{2}^{2}}{2t}}-1)
\psi](x_{1},\mu_{2})\|_{L_{\mu_{2}}^{2}}\|_{L_{x_{1}}^{2}}
\nonumber\\
&=&\|
(e^{\frac{ix_{2}^{2}}{2t}}-1)
\psi\|_{L_{x}^{2}}\nonumber\\
&\le& Ct^{-1}\|
\psi\|_{H_{x}^{0,2}}.
\label{r14}
\end{eqnarray}
Collecting (\ref{u1}), (\ref{u2}), (\ref{l}), 
(\ref{r11}), (\ref{u7}), (\ref{r12}) and (\ref{r14}), 
we obtain the desired result. 
\end{proof}


To prove Theorem \ref{nonlinear1}, 
we employ the decay estimate and the Strichartz 
estimate for 
the linear fourth order Schr\"{o}dinger equation 
(\ref{4LS}). 

\begin{lemma}\label{S} 
Let $W(t)$ be given by (\ref{os}). 

\vskip2mm
\noindent
(i) Let $2\le p\le\infty$. Then, the inequality 
\begin{equation*}
\|\langle\pt_{x_{1}}\rangle^{1-\frac{2}{p}}W(t)\psi\|_{L_{x}^{p}}
\le Ct^{-d(\frac12-\frac1p)}\|\psi\|_{L_{x}^{p'}}
\end{equation*}
holds.

\vskip2mm
\noindent
(ii) Let $(q_{j},r_{j})$ ($j=1,2$) satisfy 
$1/q_{j}+1/r_{j}=1/2$ and $2\le r_{j}<\infty$. Then, the inequality 
\begin{equation*}
\left\|\langle\pt_{x_{1}}\rangle^{\frac{1}{q_{1}}}
\int_t^{+\infty}W(t-t')F(t')dt'
\right\|_{L_t^{q_{1}}(t,\infty;L_{x}^{r_{1}})}
\le C
\|\langle\pt_{x_{1}}\rangle^{-\frac{1}{q_{2}}}
F\|_{L_t^{q_{2}'}(t,\infty;L_{x}^{r_{2}'})} 
\end{equation*}
holds.
\end{lemma}

\begin{proof}[Proof of Lemma \ref{S}.] 
See 
\cite[Theorem 3.1, Theorem 3.2]{KPV} for instance. 
\end{proof}



\section{Proof of Theorem \ref{nonlinear1}.} \label{sec:nonlinear}

In this section we prove Theorem \ref{nonlinear1}. To this end, 
we show the following lemma for the asymptotic profile.

\begin{lemma}\label{nl1} 
Let $S_{+}$ be given by (\ref{u}). Then 
we have for $t\ge 3$,
\begin{eqnarray*}
\|\hat{\psi}_{+}e^{iS_{+}(t,\xi)}\|_{H_{\xi}^{2}}
&\le& (\log t)^{2}P(\|\psi_{+}\|_{H^{0,2}}),\\
\|\frac{1}
{\sqrt{3\xi_{1}^{2}+1}}
|\hat{\psi}_{+}|\hat{\psi}_{+}e^{iS_{+}(t,\xi)}\|_{H_{\xi}^{2}}
&\le& (\log t)^{2}P(\|\psi_{+}\|_{H^{0,2}}),
\end{eqnarray*}
where $P(\|\psi_{+}\|_{H^{0,2}})$ is a polynomial in $\|\psi_{+}\|_{H^{0,2}}$ 
without constant term. 
\end{lemma}

\begin{proof}[Proof of Lemma \ref{nl1}.] 
Since the proof follows from a direct calculations, we 
omit the detail. \end{proof}

Let us start the proof of Theorem \ref{nonlinear1}. 
We first rewrite (\ref{FSP2}) as the integral equation. 
Let ${{\mathcal L}}=i\pt_t+(1/2)\Delta-(1/4)\pt_{x_{1}}^4$ 
and let
\begin{eqnarray}
w(t,\xi)=\hat{\psi}_+(\xi)e^{iS_+(t,\xi)}, \label{721}
\end{eqnarray}
where $S_+$ is given by (\ref{u}). 
From (\ref{FSP2}) and (\ref{721}), we obtain
\begin{eqnarray}
i\pt_t({{{{\mathcal F}}}}W(-t)u)&=&
{{{{\mathcal F}}}}W(-t){{{\mathcal L}}}u=\lambda
{{{{\mathcal F}}}}W(-t)|u|u,\label{741}\\
i\pt_tw&=&
\lambda
\frac{t^{-1}}
{\sqrt{3\xi_{1}^{2}+1}}|\hat{\psi}_{+}(\xi)|
\hat{\psi}_+(\xi)e^{iS_{+}(t,\xi)}.\label{751}
\end{eqnarray}
Subtracting (\ref{751}) from (\ref{741}), we have
\begin{eqnarray}
\lefteqn{i\pt_t({{{{\mathcal F}}}}W(-t)u-w)}\nonumber\\
&=&\lambda{{{{\mathcal F}}}}W(-t)
\left[|u|u
-W(t){{{\mathcal F}}}^{-1}
\left[\frac{t^{-1}}
{\sqrt{3\xi_{1}^{2}+1}}|\hat{\psi}_{+}(\xi)|
\hat{\psi}_+(\xi)e^{iS_{+}(t,\xi)}\right]\right].
\nonumber\\
\label{761}
\end{eqnarray} 
Proposition \ref{linear} and Lemma \ref{nl1} yield
\begin{eqnarray*}
W(t){{{\mathcal F}}}^{-1}
\left[
\frac{t^{-1}}
{\sqrt{3\xi_{1}^{2}+1}}|\hat{\psi}_{+}(\xi)|
\hat{\psi}_+(\xi)e^{iS_{+}(t,\xi)}
\right]
=|u_{+}|u_{+}+R_{1}(t),
\end{eqnarray*}
where $u_{+}$ is given by (\ref{u}) and $R_{1}$ satisfies 
\begin{eqnarray}
\|R_{1}(t)\|_{L_{x}^{2}}
&\le& Ct^{-1-\beta}
\|{{\mathcal F}}^{-1}
[\frac{1}
{\sqrt{3\xi_{1}^{2}+1}}|\hat{\psi}_{+}(\xi)|
\hat{\psi}_+(\xi)e^{iS_{+}(t,\xi)}]\|_{H_{x}^{0,2}}
\nonumber\\
&=& Ct^{-1-\beta}
\|\frac{1}
{\sqrt{3\xi_{1}^{2}+1}}
|\hat{\psi}_{+}|\hat{\psi}_{+}e^{iS_{+}(t,\xi)}\|_{H_{\xi}^{2}}
\nonumber\\
&\le& Ct^{-1-\beta}(\log t)^{2}P(\|\psi_{+}\|_{H_{x}^{0,2}}),
\label{801}
\end{eqnarray}
where $0<\beta<3/4$. 
Furthermore, by Proposition \ref{linear} and Lemma \ref{nl1}, 
\begin{eqnarray}
\lefteqn{
W(t){{{\mathcal F}}}^{-1}
\left[\frac{t^{-1}}
{\sqrt{3\xi_{1}^{2}+1}}|\hat{\psi}_{+}(\xi)|
\hat{\psi}_+(\xi)e^{iS_{+}(t,\xi)}\right]
}\qquad\qquad\qquad{}
\nonumber\\
&=&
|W(t){{{{\mathcal F}}}}^{-1}w|W(t){{{{\mathcal F}}}}^{-1}w+R_{1}(t)+R_{2}(t),
\label{800}
\end{eqnarray}
where 
\begin{eqnarray}
\|R_{2}(t)\|_{L_{x}^{2}}
&=&
\||W(t){{{{\mathcal F}}}}^{-1}w|W(t){{{{\mathcal F}}}}^{-1}w
-|u_{+}|u_{+}\|_{L_{x}^{2}}\nonumber\\
&\le&(\|W(t){{{{\mathcal F}}}}^{-1}w\|_{L_{x}^{\infty}}
+\|u_{+}\|_{L_{x}^{\infty}})
\|W(t){{{{\mathcal F}}}}^{-1}w-u_{+}\|_{L_{x}^{2}}\nonumber\\
&\le&Ct^{-1-\beta}(\log t)^{4}P(\|\psi_{+}\|_{H_{x}^{0,2}}).
\label{802}
\end{eqnarray}
Substituting (\ref{800}) into (\ref{761}), we obtain 
\begin{eqnarray*}
\lefteqn{i\pt_t({{{{\mathcal F}}}}W(-t)u-w)}\\
&=&\lambda{{{{\mathcal F}}}}W(-t)[|u|u-|W(t){{{{\mathcal F}}}}^{-1}w|
W(t){{{{\mathcal F}}}}^{-1}w]
-\lambda{{{{\mathcal F}}}}W(-t)(R_{1}+R_{2}).
\end{eqnarray*}
Integrating the above equation with respect to $t$ variable on $(t,\infty)$,
we have
\begin{eqnarray}
\lefteqn{u(t)-
W(t){{{\mathcal F}}}^{-1}w}\nonumber\\
&=&i\lambda\int_t^{+\infty}
W(t-\tau)[|u|u-|W(t){{{{\mathcal F}}}}^{-1}w|
W(t){{{{\mathcal F}}}}^{-1}w](\tau)d\tau\nonumber\\
& &-i\lambda\int_t^{+\infty}W(t-\tau)(R_{1}+R_{2})(\tau)d\tau.\label{811}
\end{eqnarray}
To show the existence of $u$ satisfying (\ref{811}), we 
shall prove that if $\|\psi_{+}\|_{H^{0,2}}$ is sufficiently small, 
then the map $\Phi$ given by 
\begin{eqnarray*}
\Phi[u](t)&=&i\lambda\int_t^{+\infty}
W(t-\tau)[|u|u-|W(t){{{{\mathcal F}}}}^{-1}w|
W(t){{{{\mathcal F}}}}^{-1}w](\tau)d\tau\nonumber\\
& &-i\lambda\int_t^{+\infty}W(t-\tau)(R_{1}+R_{2})(\tau)d\tau
\end{eqnarray*}
is a contraction on
\begin{eqnarray*}
{{{\bf X}}}_{\rho,T}&=&
\{u\in C([T,\infty);L^2(\rre^{2}))\cap \langle\pt_{x_{1}}
\rangle^{-\frac{1}{4}}L_{loc}^{4}(T,\infty;L^{4}(\rre^{2}));\\
& &\qquad\qquad\qquad\qquad\qquad\qquad\qquad\qquad
\|u-W(t){{\mathcal F}}^{-1}w\|_{{{\bf X}}_T}\le \rho\},\\
\|v\|_{{{\bf X}}_T}
&=&\sup_{t\ge T}
t^{\alpha}(\|v\|_{L^{\infty}(t,\infty;L_x^2)}
+\|\langle\pt_{x_{1}}\rangle^{\frac{1}{4}}v\|_{L^{4}(t,\infty;L_x^{4})})
\end{eqnarray*}
for some $T\ge3$ and $\rho>0$. 

Let $v(t)=u(t)-W(t){{\mathcal F}}^{-1}w$ and $v\in{{\bf X}}_{\rho,T}$. 
Then 
the Strichartz estimate (Lemma \ref{S}) implies
\begin{eqnarray}
\lefteqn{\|\Phi[u]-W(t){{\mathcal F}}^{-1}w
\|_{L^{\infty}(t,\infty;L_x^2)}
+\|\langle\pt_{x_{1}}\rangle^{\frac{1}{4}}
(\Phi[u]-W(t){{\mathcal F}}^{-1}w)\|_{L^4(t,\infty;L_x^{4})}}
\qquad\qquad\nonumber\\
&\le&C(\||v|v\|_{L^{\frac43}(t,\infty;L_x^{\frac43})}
+\||W(t){{\mathcal F}}^{-1}w|v\|_{L^1(t,\infty;L_x^2)})\nonumber\\
& &
+\|R_{1}\|_{L^1(t,\infty;L_x^2)})+\|R_{1}\|_{L^1(t,\infty;L_x^2)}).
\label{821}
\end{eqnarray}
By the H\"{o}lder inequality, 
\begin{eqnarray*}
\||v|v\|_{L^{\frac43}(t,\infty;L_x^{\frac43})}
&\le&C\|\|v\|_{L_x^{2}}\|v\|_{L_x^{4}}\|_{L^{\frac43}(t,\infty)}\\
&\le&C\rho\|t^{-\alpha}
\|v\|_{L_x^{4}}\|_{L^{\frac43}(t,\infty)}\\
&\le& C\rho\|t^{-\alpha}\|_{L^{2}(t,\infty)}
\|v\|_{L^{4}(t,\infty;L^{4})}\\
&\le& C\rho^2 t^{-2\alpha+\frac{1}{2}},\nonumber\\
\||W(t){{\mathcal F}}^{-1}w|v\|_{L^1(t,\infty;L_x^2)}
&\le&\|\|W(t){{\mathcal F}}^{-1}w\|_{L_x^{\infty}}\|v\|_{L_x^2}\|_{L^1(t,\infty)}\\
&\le& C\rho P(\|\psi_{+}\|_{H_{x}^{0,2}})\|t^{-1-\alpha}\|_{L^1(t,\infty)}\\
&\le&C\rho P(\|\psi_{+}\|_{H_{x}^{0,2}})t^{-\alpha}.
\end{eqnarray*} 
Substituting the above two inequalities, 
(\ref{801}), and (\ref{802}) 
into (\ref{821}), we have
\begin{eqnarray*}
\|\Phi[u]-W(t){{\mathcal F}}^{-1}w\|_{{{\bf X}}_T}
\le C(\rho^{2}T^{-\alpha+\frac{1}{2}}+\rho P(\|\psi_{+}\|_{H_{x}^{0,2}})
+T^{\alpha-\beta}(\log T)^{4}).
\end{eqnarray*}
Choosing $1/2<\alpha<\beta<3/4$, 
$T$ large enough, and $\|\psi_{+}\|_{H_{x}^{0,2}}$ sufficiently small, 
we find that $\Phi$ is a map onto ${{\bf X}}_{\rho,T}$. 
In a similar way we can conclude that $\Phi$ is a contraction 
map on ${{\bf X}}_{\rho,T}$. Therefore, by the Banach fixed point theorem 
we find that $\Phi$ has a unique fixed point in ${{\bf X}}_{\rho,T}$ 
which is the solution to the final state problem (\ref{FSP2}). 

From (\ref{FSP2}), we obtain
\begin{eqnarray}
u(t)=W(t-T)u(T)-i\lambda\int_T^tW(\tau)|u|u(\tau)d\tau.
\label{5.19}
\end{eqnarray}
Since $u(T)\in L_x^{2}(\rre^{3})$, 
combining the argument by \cite{Tsutsumi} 
with the Strichartz estimate (Lemma \ref{S}) 
and $L^2$ conservation law for (\ref{5.19}), 
we can prove that (\ref{5.19}) has a unique 
global solution in $C(\rre;L_x^2(\rre^{2}))\cap
\langle\pt_{x_{1}}\rangle^{-1/4}L_{loc}^4(\rre;L_x^{4}(\rre^{2}))$. 
Therefore the solution $u$ of (\ref{FSP2}) can be extended to all times. 

Finally we show that the solution to (\ref{FSP2}) converges to 
$u_{+}$ in $L^{2}$ as $t\to\infty$. 
Since $u-W(t){{{{\mathcal F}}}}^{-1}w\in{{\bf X}}_{\rho,T}$, 
Proposition \ref{linear} and Lemma \ref{nl1} yield
\begin{eqnarray*}
\lefteqn{\|u(t)-u_{+}(t)\|_{L_{x}^{2}}}\\
&\le&
\|u(t)-W(t){{{{\mathcal F}}}}^{-1}w\|_{L_{x}^{2}}
+\|W(t){{{{\mathcal F}}}}^{-1}w-u_{+}\|_{L_{x}^{2}}\\
&\le&Ct^{-\alpha}+Ct^{-\beta}(\log t)^{2}\\
&\le&Ct^{-\alpha},
\end{eqnarray*}
where $1/2<\alpha<\beta<3/4$. 
This completes the proof of Theorem \ref{nonlinear1}. $\qed$


\section{Proof of Theorem \ref{nonlinear3}.} \label{sec:nonlinear}

In this section we prove Theorem \ref{nonlinear3} via the argument by 
Glassey \cite{G}. To prove Theorem \ref{nonlinear3}, we employ the 
asymptotic formula \cite[Proposition 2.1]{SS} as $t\to\infty$ 
for the solution to 
\begin{eqnarray}
\left\{
\begin{array}{l}
\displaystyle{
i\pt_tu+\frac12\Delta u-\frac{1}{4}\pt_{x_{1}}^4u=0,
\qquad t>0,x\in\rre^{d},
}\\
\displaystyle{u(0,x)=\psi(x),\qquad x\in\rre^{d}.}
\end{array}
\right.
\label{4LS2}
\end{eqnarray}

\begin{lemma}\label{linear2} 
Let $W(t)\psi$ be a solution to (\ref{4LS2}). Then we have
\begin{eqnarray*}
[W(t)\psi](x)=
\frac{t^{-\frac{d}{2}}}{\sqrt{3\mu_{1}^{2}+1}}\hat{\psi}(\mu)
e^{\frac34it\mu_{1}^{4}+\frac12it|\mu|^{2}
-i\frac{d}{4}\pi}+R(t,x)
\end{eqnarray*}
for $t\ge2$,
where $\mu=(\mu_{1},\mu_{\perp})$ is given by 
\begin{equation*}
\begin{aligned}
\mu_{1}&=
\left\{\frac1{2t}\left(x_{1}+\sqrt{x_{1}^{2}+\frac{4}{27}t^{2}}\right)
\right\}^{1/3}
+
\left\{\frac1{2t}\left(x_{1}-\sqrt{x_{1}^{2}+\frac{4}{27}t^{2}}\right)
\right\}^{1/3},\\
\mu_{\perp}&=\frac{x_{\perp}}{t}.
\end{aligned}
\end{equation*}
and $R$ satisfies
\begin{eqnarray*}
\|
R(t)\|_{L_x^{p}}
\le Ct^{-d(\frac12-\frac1p)-\beta}\|\psi\|_{H_{x}^{0,s}},
\end{eqnarray*}
for $2\le p\le\infty$, $1/(4p)<\beta<1/2$ and $s>d/2-(d-1)/p+1$.
\end{lemma}

\begin{proof}[Proof of Lemma \ref{linear2}] 
See \cite[Proposition 2.1]{SS}. 
\end{proof}

\begin{proof}[Proof of Theorem \ref{nonlinear3}] 
Let $0<s<t$ and let $u$ be a solution to (\ref{4NLS}) satisfying (\ref{asym}). 
We denote $\langle u,v\rangle=\int u(x)\overline{v(x)}dx$. Then a direct calculation shows
\begin{eqnarray}
\lefteqn{\langle W(-t)u(t)-W(-s)u(s),\psi_{+}\rangle_{L_{x}^{2}}}
\nonumber\\
&=&
-i\lambda\int_{s}^{t}
\langle W(-\tau)|u|^{p-1}u(\tau),\psi_{+}\rangle_{L_{x}^{2}}d\tau
\nonumber\\
&=&
-i\lambda\int_{s}^{t}
\langle (|u|^{p-1}u)(\tau),W(\tau)\psi_{+}\rangle_{L_{x}^{2}}d\tau
\nonumber\\
&=&
-i\lambda\int_{s}^{t}
\langle (|u_{+}^{0}|^{p-1}u_{+}^{0})(\tau),u_{+}^{0}(\tau)\rangle_{L_{x}^{2}}d\tau
\nonumber\\
& &
-i\lambda\int_{s}^{t}
\langle (|u|^{p-1}u)(\tau),W(\tau)\psi_{+}-u_{+}^{0}(\tau)\rangle_{L_{x}^{2}}d\tau
\nonumber\\
& &
-i\lambda\int_{s}^{t}
\langle (|u|^{p-1}u)(\tau)-(|u_{+}^{0}|^{p-1}u_{+}^{0})(\tau),
u_{+}^{0}(\tau)\rangle_{L_{x}^{2}}d\tau
\nonumber\\
&=:&I_{1}(t,s)+I_{2}(t,s)+I_{3}(t,s),
\label{n1}
\end{eqnarray}
where $u_{+}^{0}$ is defined by 
\begin{eqnarray*}
u_{+}^{0}(t,x)=
\frac{t^{-\frac{d}{2}}}{\sqrt{3\mu_{1}^{2}+1}}\hat{\psi}_{+}(\mu)
e^{\frac34it\mu_{1}^{4}+\frac12it|\mu|^{2}
-i\frac{d}{4}\pi}.
\end{eqnarray*}
By the definition of $u_{+}^{0}$, we easily see 
\begin{eqnarray}
I_{1}(t,s)=-i\lambda
\left(\int_{\rre^{d}}
\frac{|\hat{\psi}_{+}(\mu)|^{p+1}}{(1+3\mu_{1}^{2})^{\frac{p-1}{2}}}d\mu\right)
\left(\int_{s}^{t}\tau^{-\frac{d}{2}(p-1)}d\tau\right).
\label{n2}
\end{eqnarray}
By Lemma \ref{linear2} and the conservation law for $L^{2}$ norm of $u$, 
we have
\begin{eqnarray}
|I_{2}(t,s)|
&\le&
C\int_{s}^{t}\|u(\tau)\|_{L_{x}^{2}}^{p}
\|W(\tau)\psi_{+}-u_{+}^{0}(\tau)\|_{L_{x}^{\frac{2}{2-p}}}d\tau
\nonumber\\
&\le&
C\|u_{0}\|_{L_{x}^{2}}^{p}\|\psi_{+}\|_{H_{x}^{0,s}}\int_{s}^{t}
\tau^{-\frac{d}{2}(p-1)-\beta}d\tau,
\label{n3}
\end{eqnarray}
where $(2-p)/8<\beta<1/2$ and $s>(4-d)/2+(d-1)p/2$.

By Lemma \ref{linear2} and the conservation law for $L^{2}$ norm of $u$, 
we have
\begin{eqnarray}
|I_{3}(t,s)|
&\le&
C\int_{s}^{t}(\|u(\tau)\|_{L_{x}^{2}}^{p-1}+\|u_{+}^{0}(\tau)\|_{L_{x}^{2}}^{p-1})
\|u(\tau)-u_{+}^{0}\|_{L_{x}^{2}}
\|u_{+}^{0}\|_{L_{x}^{\frac{2}{2-p}}}d\tau
\nonumber\\
&\le&
C(\|u_{0}\|_{L_{x}^{2}}^{p-1}+\|\psi_{+}\|_{L_{x}^{2}}^{p-1})
\|\langle\pt_{x_{1}}\rangle^{-(p-1)}\psi_{+}\|_{L^{\frac{p}{2}}}
\nonumber\\
& &\qquad\times\int_{s}^{t}
(\|u(\tau)-W(\tau)\psi_{+}\|_{L_{x}^{2}}+t^{-\beta}\|\psi_{+}\|_{H^{0,s}})
\tau^{-\frac{d}{2}(p-1)}d\tau,\nonumber\\
\label{n4}
\end{eqnarray}
where $1/8<\beta<1/2$ and $s>3/2$. 
By (\ref{n1}), (\ref{n2}), (\ref{n3}) and (\ref{n4}), we have
\begin{eqnarray*}
\lefteqn{|\langle W(-t)u(t)-W(-s)u(s),\psi_{+}\rangle_{L_{x}^{2}}|}
\\
&\ge&|\lambda|
\left(\int_{\rre^{d}}
\frac{|\hat{\psi}_{+}(\mu)|^{p+1}}{(1+3\mu_{1}^{2})^{\frac{p-1}{2}}}d\mu\right)
\left(\int_{s}^{t}\tau^{-\frac{d}{2}(p-1)}d\tau\right)\\
& &-C\int_{s}^{t}\tau^{-\frac{d}{2}(p-1)-\beta}d\tau
-C\int_{s}^{t}\|u(\tau)-W(\tau)\psi_{+}\|_{L_{x}^{2}}\tau^{-\frac{d}{2}(p-1)}d\tau.
\end{eqnarray*}
Hence by the assumption (\ref{asym}) on $u$, 
we see that there exists $T>0$ such that 
for $t>s>T$, 
\begin{eqnarray*}
\lefteqn{\left|\langle W(-t)u(t)-W(-s)u(s),\psi_{+}\rangle_{L_{x}^{2}}\right|}
\\
&\ge&\frac{|\lambda|}{2}\left(\int_{\rre^{d}}
\frac{|\hat{\psi}_{+}(\mu)|^{p+1}}{(1+3\mu_{1}^{2})^{\frac{p-1}{2}}}d\mu
\right)
\left(\int_{s}^{t}\tau^{-\frac{d}{2}(p-1)}d\tau\right).
\end{eqnarray*}
Hence we have $u\equiv 0$. This completes the proof. 
\end{proof}

\vskip3mm
\noindent {\bf Acknowledgments.} 
Part of this work was done while J.S was 
visiting the Department of Mathematics at 
Universit\'e de Paris-Sud, Orsay 
whose hospitality he gratefully acknowledges.
J.S. is partially supported by JSPS, 
Grant-in-Aid for Scientific Research (B) 17H02851. 
J.-C. S. is partly supported by the ANR project ANuI.

\end{document}